\documentclass[11pt]{amsart}

\usepackage{amscd}
\usepackage{amsmath, amssymb}
\usepackage{amsfonts}
\newcommand{\de}{\partial}

\newcommand{\ddbar}{\sqrt{-1} \partial \overline{\partial}}

\newcommand{\ov}[1]{\overline{#1}}
\newcommand{\mn}{\sqrt{-1}}

\newcommand{\ti}[1]{\tilde{#1}}
\newcommand{\vp}{\varphi}

\newcommand{\ve}{\varepsilon}

\renewcommand{\leq}{\leqslant}
\renewcommand{\geq}{\geqslant}
\renewcommand{\le}{\leqslant}
\renewcommand{\ge}{\geqslant}

\newcommand{\be}{\begin{equation}}
\newcommand{\ee}{\end{equation}}

\begin{document}
\newtheorem{claim}{Claim}
\newtheorem{theorem}{Theorem}[section]
\newtheorem{lemma}[theorem]{Lemma}
\newtheorem{corollary}[theorem]{Corollary}
\newtheorem{proposition}[theorem]{Proposition}
\newtheorem{question}{question}[section]
\theoremstyle{definition}
\newtheorem{remark}[theorem]{Remark}

\numberwithin{equation}{section}

\title[$C^{1,1}$ regularity of geodesics]{On the $C^{1,1}$ regularity of geodesics in the space of K\"ahler metrics}
\author[J. Chu]{Jianchun Chu}
\address{School of Mathematical Sciences, Peking University, Yiheyuan Road 5, Beijing, P.R.China, 100871}
\author[V. Tosatti]{Valentino Tosatti}
\address{Department of Mathematics, Northwestern University, 2033 Sheridan Road, Evanston, IL 60208}
\author[B. Weinkove]{Ben Weinkove}
\address{Department of Mathematics, Northwestern University, 2033 Sheridan Road, Evanston, IL 60208}

\begin{abstract}
We prove that any two K\"ahler potentials on a compact K\"ahler manifold can be connected by a geodesic segment of $C^{1,1}$ regularity. This follows from an {\em a priori} interior real Hessian bound for solutions of the nondegenerate complex Monge-Amp\`ere equation, which is independent of a positive lower bound for the right hand side.
\end{abstract}

\maketitle

\section{Introduction} \label{sectionintro}

Let $(X^m, g)$ be a compact $m$-dimensional K\"ahler manifold without boundary.  Write $\omega = \sqrt{-1} g_{i\ov{j}}dz^i \wedge d\ov{z}^j$ for the K\"ahler form of $g$.  Write $\mathcal{H}_{\omega}$ for the space of K\"ahler metrics cohomologous to $\omega$, which we identify with   smooth functions $\varphi$ with $\omega + \ddbar \varphi>0$, modulo constants.  The space $\mathcal{H}_{\omega}$ can be endowed with the structure of an infinite dimensional Riemannian manifold \cite{Ma, Se, Do}. Chen \cite{Ch} showed that any two potentials $\varphi_0, \varphi_1 \in \mathcal{H}_{\omega}$ can be joined by a weak geodesic segment $\{ \varphi_t \}_{t\in [0,1]}$.  The purpose of this note is to explain how the arguments of our earlier paper \cite{CTW} on the Monge-Amp\`ere equation imply a $C^{1,1}$ regularity result for these geodesics.

We first recall the equivalence of the geodesic equation with the homogeneous complex Monge-Amp\`ere equation \cite{Se, Do}.  Given a smooth family of K\"ahler potentials $\{ \varphi_t \}_{t\in [0,1]}$, we let $\Sigma=\{z\in\mathbb{C}\ |\ 1\leq |z|\leq e\}$ and we define a function $\Phi$ on $X \times \Sigma$ by $\Phi(x,z)=\vp_{\log|z|}(x)$.
Then, as noted in \cite{Se,Do}, the path $\{\vp_t\}$ is a geodesic connecting $\vp_0$ and $\vp_1$ if and only if $\Phi$
solves the homogeneous complex Monge-Amp\`ere equation
\begin{equation}\label{ma3}
(\pi^*\omega+\ddbar\Phi)^{m+1}=0,
\end{equation}
with boundary data $\Phi(x,1)=\vp_0(x), \Phi(x,e)=\vp_1(x)$, $x\in X$. From \eqref{ma3} it also follows easily that $\pi^*\omega+\ddbar\Phi\geq 0$ on $X\times\Sigma$, where $\pi:X\times\Sigma\to X$ is the projection. Generalizing this fact, we say that a bounded function $\Phi$ on $X\times\Sigma$ is a weak geodesic connecting $\vp_0$ and $\vp_1$
if $\pi^*\omega+\ddbar\Phi\geq 0$ weakly on $X\times\Sigma$ and $\Phi$ solves \eqref{ma3} in the sense of Bedford-Taylor \cite{BT} with the same boundary data as above. Chen \cite{Ch} proved that there is a unique such weak geodesic $\Phi$, and that the quantities $\sup_{X\times\Sigma}|\Phi|,$ $\sup_{X\times\Sigma}|\nabla\Phi|,$ $\sup_{X\times\Sigma}\Delta \Phi$ and $\sup_{\de(X\times\Sigma)}|\nabla^2\Phi|$ are all bounded (see also \cite{B2,CKNS,Gu,Ya}), so in particular $\Phi$ is in $C^{1,\alpha}(X\times\Sigma)$ for all $0<\alpha<1$. Here and in the following, when we say that a function belongs to a certain function space on a manifold with boundary, we mean that it has the stated regularity up to (and including) the boundary.

It was expected that the real Hessian of $\Phi$ is bounded in the interior (namely that $\sup_{X \times \Sigma} |\nabla^2 \Phi| \le C)$.  This would imply that geodesics are $C^{1,1}$.  Some progress towards this was made recently:  B\l ocki \cite{B2} proved that $\Phi$ is  in $C^{1,1}(X\times\Sigma)$ provided that $(X,g)$ has nonnegative bisectional curvature, while Berman \cite{Be} proved that the restrictions to $X$ given by $\vp_t=\Phi(\cdot,e^t), 0\leq t\leq 1,$ are all in $C^{1,1}(X)$ provided that $[\omega]=c_1(L)$ for some holomorphic line bundle $L$ over $X$. However, the problem of proving the full $C^{1,1}$ regularity of weak geodesics remained open in general.  Our main result resolves this completely.

\begin{theorem}\label{geodesic}
Given any compact K\"ahler manifold $(X,g)$ and any two K\"ahler potentials on it, the weak geodesic $\Phi$ connecting them belongs to $C^{1,1}(X\times\Sigma)$.
\end{theorem}

This result should be contrasted with the negative results of Lempert-Vivas \cite{LV}, Darvas-Lempert \cite{DL} and Darvas \cite{Da}, which show that in general the weak geodesic satisfies that $\ddbar\Phi\not\in C^0(X\times\Sigma)$. See also \cite{BB, CLP, CT, Do2, He, PSS,PS,PS4,RN,RN2} and references therein for other recent developments on issues related to weak geodesics and their regularity. Our result can also be used to simplify part of the arguments in \cite{BB}, and possibly to extend the results in \cite{RN2} to more general manifolds.

 Theorem \ref{geodesic} is a consequence of the following  general interior estimate for the complex Monge-Amp\`ere equation.  Let  $(M^n,g)$  now be a compact $n$-dimensional K\"ahler manifold with possibly nonempty boundary $\de M$.  Again, write $\omega$ for the K\"ahler form of $g$.  Suppose that $\vp\in C^\infty(M,\mathbb{R})$ satisfies $\omega+\ddbar\vp>0$ and solves the complex Monge-Amp\`ere equation
\begin{equation}\label{ma2}
(\omega+\ddbar\vp)^n=e^F\omega^n,
\end{equation}
for some smooth function $F$.

We have the following interior {\em a priori} $C^{1,1}$ estimate, which is proved in Section \ref{sectionproof}:
\begin{theorem}\label{secondord}
Let $\varphi$ solve the complex Monge-Amp\`ere equation (\ref{ma2}).  Then there exists a constant $C$, depending only on $(M,g)$, on upper bounds on $\|\vp\|_{C^1(M,g)}$, $\sup_M \Delta_g\vp$,  $\sup_{\de M}|\nabla^2\vp|_g$ (if $\de M\neq\emptyset$), $\sup_M |\nabla F|_g,$ and on a lower bound on $\nabla^2 F$ (with respect to $g$) such that
$$\sup_M| \nabla^2 \varphi |_g \le C.$$
\end{theorem}
Crucially, our estimate is independent of $\inf_M F$, and so it can be applied to the homogeneous complex Monge-Amp\`ere equation to obtain regularity for geodesics in the space of K\"ahler metrics. This is explained in Section \ref{sectionthm1}.

More generally, combining Theorem \ref{secondord} with \cite[Theorem 1.4]{B2} and \cite[Theorem B]{Bo} we immediately obtain the following improvement of these results:
\begin{corollary}\label{geodesic2}
Let $(M,g)$ be a compact K\"ahler manifold with nonempty boundary, which we assume is weakly pseudoconcave. Given a smooth function $\vp_0$ on $M$ with $\omega+\ddbar\vp_0>0$, there exists a unique solution $\vp\in C^{1,1}(M)$ of
$$(\omega+\ddbar\vp)^n=0, \textrm{ on }M,\quad \vp=\vp_0,\textrm{ on }\de M.$$
\end{corollary}

As alluded to above, the proof of Theorem \ref{secondord} is essentially contained in our earlier paper \cite{CTW}  where we solved the Monge-Amp\`ere equation on compact almost-complex manifolds (see also \cite{Pl}). Indeed, the reader may easily verify, combining the arguments in \cite{CTW} with the small modifications in this paper, we also obtain that the second order estimate in the general almost-complex setting of \cite[Proposition 5.1]{CTW}  does not depend on $\inf_M F$. This includes  the Hermitian case, when $J$ is integrable but $d\omega\neq 0$, which was studied for example in \cite{GL,TW}. The reader who is interested in the most general setup is referred to \cite{CTW}.

The arguments of \cite{CTW} are long and intricate because of technicalities that arise when the complex structure is not integrable, and/or when $d\omega\neq 0$, and because there we  do not  \emph{a priori} have a bound on the \emph{complex} Hessian of $\varphi$. The aim of this note is to give a self-contained proof of Theorem \ref{secondord} which is substantially simpler and shorter than the arguments of \cite{CTW}.

We discuss briefly the idea of the proof.  If one follows the approach of B\l ocki \cite{B2}, a difficulty arises when applying the maximum principle to a quantity involving the largest eigenvalue of the real Hessian, given by
\begin{equation}\label{max}
\sup_{|V|_g=1}\nabla^2\vp(V,V).
\end{equation}
A perturbation argument is needed to make this smooth near a maximum point $x_0$. In \cite{B2} this is done by taking a unit vector $V_0$ which maximizes \eqref{max} at $x_0$, extending it smoothly to a unit vector $V$ in a neighborhood of $x_0$, and considering the smooth quantity $\nabla^2\vp(V,V)$ near $x_0$. To deal with bad terms involving the trace of $g_{i\ov{j}}$ with respect to $g_{i\ov{j}} + \varphi_{i\ov{j}}$, one is forced to consider the logarithm of this quantity, which introduces bad third order terms that cannot be controlled, unless $g$ has nonnegative bisectional curvature in which case there is no need to take the logarithm \cite{B2}. Instead, we use a different perturbation argument as in \cite{CTW,S,STW}, replacing \eqref{max} near $x_0$ with the largest eigenvalue function of a small perturbation of $\nabla^2\vp$, so that its largest eigenspace is $1$-dimensional and therefore the largest eigenvalue varies smoothly near $x_0$. Crucially, this gives new good third order terms (the terms in the first sum on the right hand side of \eqref{LhatQ4}), which are shown, via a series of delicate calculations and estimates, to be just barely enough to control the bad terms.

\begin{remark}
In a follow up paper \cite{CTW2}, we will investigate the regularity of weak geodesic \emph{rays} which arise from test configurations, as in the work of Phong-Sturm \cite{PS, PS2,PS3,PS4} who established $C^{1,\alpha}$ regularity in this setting.
\end{remark}

{\bf Acknowledgments. }  The authors thank J. Song for pointing out a simplification of the proof given in an earlier version of this paper, and M. P\u{a}un for interesting discussions.  We also thank the referee for some helpful comments.  The first-named author would like to thank his advisor G. Tian for encouragement and support.
  The second-named author was partially supported by NSF grant DMS-1610278, and the third-named author by NSF grant DMS-1406164. This work was completed while the second-named author was visiting the Yau Mathematical Sciences Center at Tsinghua University in Beijing, which he would like to thank for the hospitality.

\section{$C^{1,1}$ bound for the complex Monge-Amp\`ere equation}\label{sectionproof}

In this section we give the proof of Theorem \ref{secondord}. We follow the approach that we introduced recently in \cite{CTW},  taking a little more care to make sure that all the estimates are independent of $\inf_M F$. Moreover, we make many simplifications, due to the fact that we assume we are in the K\"ahler case, and also because we allow our estimates to depend on $\sup_M\Delta_g\vp$ (since this quantity is known to be bounded {\em a priori} in the settings of Theorem \ref{geodesic} and Corollary \ref{geodesic2} due to previous works mentioned in the introduction), while one of the main points of \cite{CTW} is that such an estimate on the Laplacian is not available in the almost-complex case, and our real Hessian bound there does not use it (but of course it implies it).

\begin{proof}[Proof of Theorem \ref{secondord}]
Up to modifying $\vp$ by adding a constant to it, we will assume that $\sup_M\vp=0$. We apply the maximum principle to the quantity
$$Q = \log \lambda_1( \nabla^2 \varphi) + h(| \partial \varphi |^2_g)  -A \varphi,$$
where $\lambda_1( \nabla^2 \varphi)$ is the largest eigenvalue of the real Hessian $\nabla^2 \varphi$ (with respect to the metric $g$).
The function $h$ is given by
\begin{equation} \label{defnh}
h(s) = - \frac{1}{2} \log (1+ \sup_M |\partial \varphi|^2_g  - s),
\end{equation}
and $A>1$ is a constant to be determined (which will be uniform, in the sense that it will depend only on the background data and on the quantities in the statement of Theorem \ref{secondord}).  Note that $h(| \partial \varphi|^2_g)$ is uniformly bounded, and
\begin{equation} \label{proph}
\frac{1}{2}\geq h' \geq \frac{1}{2+2\sup_M |\partial \varphi|^2_g}>0, \quad \textrm{and } h'' = 2 (h')^2,
\end{equation}
where we are evaluating $h$ and its derivatives at $|\de\vp|^2_g$.
Observe that since
\begin{equation}\label{stupid}
| \nabla^2 \varphi |_g\leq C\lambda_1( \nabla^2 \varphi)+C,
\end{equation}
a simple consequence of the fact that $\omega+ \ddbar \vp>0$, it suffices to bound $Q$ from above on $M$.
 Note also that $Q$ is a continuous function on its domain $\{\lambda_1(\nabla^2 \varphi)>0\}$, and achieves a maximum at a point $x_0 \in M$ with $\lambda_1(\nabla^2 \varphi(x_0))>0$, which we may assume is not on $\de M$ (otherwise we are done).  On the other hand, $Q$ may not be smooth, since the eigenspace associated to $\lambda_1$ may have dimension strictly larger than $1$.  Because of  this, we use a perturbation argument, as in \cite{CTW,S, STW}.

Fix holomorphic normal coordinates $(z^1,\dots,z^n)$ for $g$ centered at $x_0$, and let $z^j=x^{2j-1}+\mn x^{2j}$, so that $(x^1,\dots,x^{2n})$ are real coordinates near $x_0$. In what follows, we will use Latin letters $i,j,k, \ldots$ for ``complex'' indices ranging from $1$ to $n$, and Greek letters $\alpha, \beta, \ldots$ for ``real'' indices ranging from $1$ to $2n$.
We define
$$\tilde{g}_{i\ov{j}} = g_{i\ov{j}} + \varphi_{i\ov{j}}$$ and
we may assume that at $x_0$ the matrix $(\tilde{g}_{i\ov{j}})$ is diagonal with
$$\tilde{g}_{1\ov{1}} \ge \tilde{g}_{2\ov{2}} \ge \cdots \ge \tilde{g}_{n\ov{n}}.$$

Let $V_1$ be a unit vector (with respect to $g$) corresponding to the largest eigenvalue $\lambda_1$ of $\nabla^2 \varphi$, so that at $x_0$,
$$\nabla^2\varphi(V_1,V_1) = \lambda_1.$$
We extend $V_1$ to an orthonormal basis $V_1, \ldots, V_{2n}$  of eigenvectors of $\nabla^2\vp$ with respect to $g$ at $x_0$, corresponding to eigenvalues $\lambda_1(\nabla^2\vp) \geq \lambda_2(\nabla^2\vp) \ge  \ldots \ge \lambda_{2n}(\nabla^2\vp)$. Denote by $\{ V_{\ \, \beta}^{\alpha} \}_{\alpha=1}^{2n}$ the components of the vector $V_{\beta}$ at $x_0$, with respect to the coordinates $x^1, \ldots, x^{2n}$.  Extend $V_1, V_2, \ldots, V_{2n}$ to be vector fields in a neighborhood of $x_0$  by taking the components to be constant,  noting that the  $V_{\beta}$ may only be eigenvectors for $\nabla^2 \varphi$ at $x_0$.

To avoid the inconvenient situation where $\lambda_1(\nabla^2\varphi)=\lambda_2(\nabla^2\varphi)$,
 define near $x_0$ a smooth symmetric semipositive definite section $B= (B_{\alpha \beta})$ of $T^*M \otimes T^*M$ by
$$B= B_{\alpha \beta}dx^{\alpha} \otimes dz^{\beta} =\sum_{\alpha, \beta}( \delta_{\alpha \beta} - V_{\ \, 1}^{\alpha} V_{\ \, 1}^{\beta}) dx^{\alpha} \otimes dx^{\beta},$$
and  a local endomorphism $\Phi=( \Phi_{\ \, \beta}^{\alpha})$ of $TM$ by
\begin{equation} \label{definePhi}
\Phi_{\ \, \beta}^{\alpha}=  g^{ \alpha \gamma} \nabla_{\gamma \beta}^2 \varphi -  g^{\alpha \gamma} B_{\gamma \beta}.
\end{equation}
We now consider  $\lambda_1(\Phi)$, which is smooth and satisfies $\lambda_1(\Phi) \le \lambda_1(\nabla^2 \varphi)$ in a neighborhood of $x_0$ and
 $\lambda_1(\Phi) = \lambda_1(\nabla^2 \varphi)$ at $x_0$.  The vector fields $V_1,\dots,V_{2n}$ are still eigenvectors for $\Phi$ at $x_0$, with eigenvalues $\lambda_1(\Phi) > \lambda_2(\Phi) \ge  \ldots \ge \lambda_{2n}(\Phi)$.  In what follows we will often write $\lambda_{\alpha}$ for $\lambda_{\alpha}(\Phi)$.

Define a new perturbed smooth quantity $\hat{Q}$  in a neighborhood of $x_0$ by
$$\hat{Q} = \log \lambda_1(\Phi) + h(| \partial \varphi |^2_g) -A \varphi,$$
which still attains a maximum at $x_0$.

Let $\Delta_{\ti{g}}= \tilde{g}^{i\ov{j}} \partial_i \partial_{\ov{j}}$ be the (complex) Laplacian of $\ti{g}$, where we are writing $\partial_i$, $\partial_{\ov{j}}$ for $\frac{\partial}{\partial z^i}$, $\frac{\partial}{\partial \ov{z}^j}$.  Later we may also write $\varphi_i$ for $\partial_i \varphi$ etc. We will compute $\Delta_{\ti{g}} \hat{Q}$ at $x_0$.
In the computation, we may and do assume without loss of generality that $\lambda_1>>1$ at $x_0$.  Note that by assumption $\Delta_g \varphi \le C$, so $\tilde{g}_{i\ov{i}}$ and $\varphi_{i\ov{i}}$ are uniformly bounded from above at $x_0$ and hence also
\begin{equation} \label{ag}
\ti{g}^{i\ov{i}} \ge c, \quad \textrm{for } i=1, 2, \ldots, n,
\end{equation}
for a uniform $c>0$.

\begin{lemma} \label{lemmalbl1}  At $x_0$, we have
\begin{equation} \label{LhatQ4}
\begin{split}
0 \ge \Delta_{\ti{g}} \hat{Q}
\ge {} & 2 \sum_{\alpha >1}  \frac{\tilde{g}^{i\ov{i}} |\partial_i (\varphi_{V_{\alpha} V_1})|^2}{\lambda_1(\lambda_1-\lambda_{\alpha})} + \frac{\tilde{g}^{p\ov{p}} \tilde{g}^{q\ov{q}} | V_1(\tilde{g}_{p\ov{q}})|^2}{\lambda_1} - \frac{\tilde{g}^{i\ov{i}} | \partial_i (\varphi_{V_1 V_1})|^2}{\lambda_1^2} \\
{} & + h' \sum_k \tilde{g}^{i\ov{i}} (| \varphi_{ik}|^2 + |\varphi_{i\ov{k}}|^2) + h'' \tilde{g}^{i\ov{i}} |\partial_i | \partial \varphi|^2_g|^2 \\ {} &+ (A-C) \sum_i \tilde{g}^{i\ov{i}}
 - An.
\end{split}
\end{equation}
where
$$\vp_{\alpha\beta}=\nabla^2_{\alpha\beta}\vp,\quad \varphi_{V_{\alpha} V_{\beta}} =  \varphi_{\gamma \delta} V_{\ \, \alpha}^{\gamma} V_{\ \, \beta}^{\delta} =\nabla^2\vp(V_\alpha,V_\beta).$$

\end{lemma}
\begin{proof} Compute
\begin{equation} \label{lbqh1}
\begin{split} \Delta_{\ti{g}}\hat{Q}  = {} & \frac{\Delta_{\ti{g}}(\lambda_1)}{\lambda_1} - \frac{\tilde{g}^{i\ov{i}} | \partial_i (\lambda_1)|^2}{\lambda_1^2} + h'\Delta_{\ti{g}} (| \partial \varphi|^2_g)  + h'' \tilde{g}^{i\ov{i}} |\partial_i | \partial \varphi|^2_g|^2  - A \Delta_{\ti{g}}\varphi.
\end{split}
\end{equation}
We now prove a lower bound for $\Delta_{\tilde{g}} (\lambda_1)$.  Using  the
fact that $\tilde{g}$ is diagonal at $x_0$ and  the coordinates are normal for $g$,
\begin{equation} \label{Llambda1}
\begin{split}
\Delta_{\ti{g}} (\lambda_1) = {} & \tilde{g}^{i\ov{i}} \lambda_1^{\alpha \beta, \gamma \delta} \partial_i (\Phi^{\gamma}_{\ \, \delta}) \partial_{\ov{i}} (\Phi^{\alpha}_{\ \,  \beta}) + \tilde{g}^{i\ov{i}} \lambda_1^{\alpha \beta} \partial_i \partial_{\ov{i}} (\Phi^{\alpha}_{\ \, \beta}) \\
= {} & \tilde{g}^{i\ov{i}} \lambda_1^{\alpha \beta, \gamma \delta} \partial_i (\varphi_{\gamma \delta}) \partial_{\ov{i}} (\varphi_{\alpha \beta}) + \tilde{g}^{i\ov{i}} \lambda_1^{\alpha \beta} \partial_i \partial_{\ov{i}} (\varphi_{\alpha \beta}) + \tilde{g}^{i\ov{i}} \lambda_1^{\alpha \beta} \varphi_{ \gamma \beta} \partial_i \partial_{\ov{i}} (g^{\alpha \gamma}) \\
& -  \tilde{g}^{i\ov{i}} \lambda_1^{\alpha \beta}B_{\gamma \beta} \partial_i \partial_{\ov{i}} (g^{\alpha \gamma})
 \\
\ge {} & 2 \sum_{\alpha >1} \tilde{g}^{i\ov{i}} \frac{ |\partial_i (\varphi_{V_{\alpha} V_1})|^2}{\lambda_1-\lambda_{\alpha}} + \tilde{g}^{i\ov{i}} \partial_i \partial_{\ov{i}} (\varphi_{V_1 V_1})
 - C \lambda_1 \sum_i \tilde{g}^{i\ov{i}}.
\end{split}
\end{equation}
Here we used the elementary formulas (see \cite[Lemma 5.2]{CTW}), holding at $x_0$,
\begin{equation} \label{formulae}
\begin{split}
\lambda_1^{\alpha \beta} := {} & \frac{\partial \lambda_1}{\partial \Phi^{\alpha}_{\ \, \beta}} =  V_{\ \, 1}^{\alpha} V_{\ \, 1}^{\beta} \\
\lambda_1^{\alpha \beta, \gamma \delta} := {} & \frac{\partial^2 \lambda_1}{\partial \Phi^{\alpha}_{\ \, \beta} \partial \Phi^{\gamma}_{\ \, \delta}} = \sum_{\mu >1}   \frac{V_{\ \, 1}^{\alpha} V_{\ \, \mu}^{\beta} V_{\ \, \mu}^{\gamma} V_{\ \, 1}^{\delta} + V_{\ \, \mu}^{\alpha} V_{\ \, 1}^{\beta} V_{\ \, 1}^{\gamma} V_{\ \, \mu}^{\delta}}{\lambda_1 - \lambda_{\mu}}. \end{split}
\end{equation}

Since the Christoffel symbols of the connection of $g$ vanish at $x_0$ and the components of $V_1$ are constant in our coordinate system, a short calculation shows that
\begin{equation} \label{blah0}
\begin{split}
\tilde{g}^{i\ov{i}} \partial_i \partial_{\ov{i}} (\varphi_{V_1 V_1})\ge {} &
 \tilde{g}^{i\ov{i}} V_1 V_1 (\partial_i \partial_{\ov{i}} \varphi)  - C \lambda_1 \sum_i \tilde{g}^{i\ov{i}} \\
\ge {}&   \tilde{g}^{i\ov{i}} V_1 V_1 (\tilde{g}_{i\ov{i}})  - C' \lambda_1 \sum_i \tilde{g}^{i\ov{i}},
\end{split}
\end{equation}
where we used that $|\de\vp|_g$ is uniformly bounded, and that we may assume without loss of generality that $\lambda_1(x_0)$ is large.
Applying $V_1V_1$ to the logarithm of  (\ref{ma2}),
\begin{equation} \label{logequation}
\log \det \tilde{g} = \log \det g + F,
\end{equation}
we obtain
\begin{equation} \label{blah1}
\begin{split}
\tilde{g}^{i\ov{i}} V_1 V_1 (\tilde{g}_{i\ov{i}}) = {} & \tilde{g}^{p\ov{p}} \tilde{g}^{q\ov{q}} | V_1(\tilde{g}_{p\ov{q}})|^2 + V_1V_1(\log\det g)+V_1 V_1(F).
\end{split}
\end{equation}
From this, (\ref{Llambda1}) and (\ref{blah0}) we have
\begin{equation} \label{l1lb}
\begin{split}
\Delta_{\ti{g}} (\lambda_1)  \ge {} & 2 \sum_{\alpha>1} \tilde{g}^{i\ov{i}} \frac{ |\partial _i (\varphi_{V_{\alpha} V_1})|^2}{\lambda_1 - \lambda_{\alpha}} + \tilde{g}^{p\ov{p}} \tilde{g}^{q\ov{q}} |V_1(\tilde{g}_{p\ov{q}})|^2 - C \lambda_1 \sum_i \tilde{g}^{i\ov{i}}.
\end{split}
\end{equation}
Next,  at $x_{0}$,
\begin{equation}\label{lemmapv}
\begin{split}
\Delta_{\ti{g}}(|\partial\varphi|_{g}^{2}) ={} &\sum_{k}\tilde{g}^{i\overline{i}}(|\varphi_{ik}|^{2}+| \varphi_{i\ov{k}}|^{2})
+2\textrm{Re}\left(\sum_{k}\varphi_{k}F_{\overline{k}}\right)\\
{} & +\tilde{g}^{i\overline{i}}\partial_{i}\partial_{\ov{i}}(g^{k\overline{\ell}})\varphi_{k}\varphi_{\overline{\ell}}\\
\geq {} & \sum_{k}\tilde{g}^{i\overline{i}}(|\varphi_{ik}|^{2}+|\varphi_{i\ov{k}}|^{2})- C\sum_{i}\tilde{g}^{i\overline{i}},
\end{split}
\end{equation}
where to obtain the first line we have applied $\partial_{\ov{k}}$ to  \eqref{logequation}.
Lastly, we observe that
\begin{equation}\label{triv}
\Delta_{\ti{g}}\vp=n-\sum_i\ti{g}^{i\ov{i}}.
\end{equation}
The result then follows by combining (\ref{lbqh1}), the first equation of (\ref{formulae}), (\ref{l1lb}), \eqref{lemmapv} and (\ref{triv}).
\end{proof}

We need to deal with the negative third term
\begin{equation} \label{badterm}
- \frac{\tilde{g}^{i\ov{i}} | \partial_i (\varphi_{V_1 V_1})|^2}{\lambda_1^2}
\end{equation}
on the right hand side of (\ref{LhatQ4}).

Define  a local $(1,0)$ vector field by $$W_1:=\frac{1}{\sqrt{2}}(V_1-\mn JV_1),$$
where $J$ is the complex structure.
We write at $x_0$,
\begin{equation} \label{fix}
W_1=\sum_{q=1}^n \nu_q \partial_q,\quad \sum_{q=1}^n|\nu_q|^2=1,
\end{equation}
for complex numbers $\nu_1, \ldots, \nu_n$, where the second equation follows from the fact that $W_1$ is $g$-unit at $x_0$.

Next, define $\mu_2, \mu_3, \ldots, \mu_{2n} \in \mathbb{R}$ by
\begin{equation} \label{defnmu}
JV_1=\sum_{\alpha>1} \mu_\alpha V_\alpha,\quad \sum_{\alpha>1}\mu_\alpha^2=1, \quad \textrm{at } x_0,
\end{equation}
noting that at $x_0$ the vector $JV_1$ is $g$-unit and $g$-orthogonal to $V_1$.

Since we are using complex coordinates, the complex structure $J$ in our real coordinates $(x^1,\dots,x^{2n})$ has constant coefficients, and so do the vector fields $JV_1$ and $W_1$.

Then we have:

\begin{lemma} \label{lemmauno} There is a uniform constant $C\geq 1$ such that if $0 < \ve< 1/2$ and $\lambda_1(x_0) \ge C/\ve^2,$ then at $x_0$ we have
\begin{equation}
\begin{split}
\sum_i \frac{\tilde{g}^{i\ov{i}} |\partial_i (\varphi_{V_1V_1})|^2}{\lambda_1^2} \le {} & 2(h')^2 \tilde{g}^{i\ov{i}} |\partial_i |\partial \varphi|_g^2|^2 + 4 \ve A^2 \tilde{g}^{i\ov{i}}|\varphi_i|^2 \\
{} & +2\sum_{\alpha>1} \frac{\tilde{g}^{i\ov{i}} | \partial_i (\varphi_{V_{\alpha} V_1})|^2}{\lambda_1(\lambda_1-\lambda_{\alpha})} + \frac{\tilde{g}^{p\ov{p}} \tilde{g}^{q\ov{q}} |V_1(\tilde{g}_{p\ov{q}})|^2}{\lambda_1} + \sum_i \tilde{g}^{i\ov{i}}.
\end{split}
\end{equation}
\end{lemma}
\begin{proof}
Writing $E$ for an ``error'' term which is uniformly bounded $|E| \le C,$ we use the definition of $W_1$, the $\mu_{\alpha}$ and the $\nu_q$ to obtain at $x_0$,
\[
\begin{split}
\partial_i (\varphi_{V_1 V_1}) = {} & \sqrt{2} \partial_i \varphi_{V_1  \ov{W_1}} - \sqrt{-1} \partial_i (\varphi_{V_1 J V_1}) \\
= {} & \sqrt{2} \ov{W_1} \partial_i (V_1 (\varphi)) - \sqrt{-1} JV_1 \partial_i (V_1(\varphi)) + E \\
= {} & \sqrt{2} \sum_q \ov{\nu_q} V_1 (\tilde{g}_{i\ov{q}})  - \sqrt{-1} \sum_{\alpha>1} \mu_{\alpha} \partial_i(\varphi_{V_1 V_{\alpha}}) + E,
\end{split}
\]
where we used that $V_1, JV_1$ and $W_1$ have constant coefficients, and that the coordinates are normal at $x_0$, which imply that $(\nabla_{V_1}\ov{W_1})(x_0)=0$ and so on, which shows that all the commutations of $3$ derivatives above give error terms which only involve $\de\vp$, hence are uniformly bounded.
Therefore
\begin{equation} \label{tech15}
\begin{split}
\lefteqn{(1-2\ve) \sum_{i} \frac{\tilde{g}^{i\ov{i}} | \partial_i (\varphi_{V_1V_1})|^2}{\lambda_1^2} } \\ = {} & (1-2\ve) \sum_{i} \frac{\tilde{g}^{i\ov{i}} | \sqrt{2} \sum_{q} \ov{\nu_q} V_1 (\tilde{g}_{i\ov{q}}) - \sqrt{-1} \sum_{\alpha>1} \mu_{\alpha}\partial_i(\varphi_{V_1 V_{\alpha}}) +E |^2}{\lambda_1^2} \\
\le {} &  (1-\ve) \sum_{i} \frac{\tilde{g}^{i\ov{i}} | \sqrt{2} \sum_{q} \ov{\nu_q} V_1 (\tilde{g}_{i\ov{q}}) - \sqrt{-1} \sum_{\alpha>1} \mu_{\alpha}\partial_i(\varphi_{V_1 V_{\alpha}})  |^2}{\lambda_1^2} + \frac{C}{\ve\lambda_1^2} \sum_i \tilde{g}^{i\ov{i}} \\
  \leq {} & (1-\ve)\left(1+\frac{1}{\ve}\right)\sum_{i}\frac{2\tilde{g}^{i\ov{i}}}{\lambda_1^2} \left|\sum_{q}\ov{\nu_q}V_1(\ti{g}_{i\ov{q}})\right|^2+(1-\ve^2)\sum_{i}\frac{\tilde{g}^{i\ov{i}}}{\lambda_1^2} \left|\sum_{\alpha>1}\mu_\alpha \partial_i(\vp_{V_1V_\alpha})\right|^2 \\
   {} & + \sum_i \tilde{g}^{i\ov{i}},
\end{split}
\end{equation}
using that $\lambda_1\geq C/\ve^2\geq \sqrt{C/\ve}$. Using this together with \eqref{ag} and the fact that $|\nu_q| \le 1$, and $\lambda_1\geq C/\ve$,
\begin{equation} \label{tech1}
\begin{split}
(1-\ve)\left(1+\frac{1}{\ve}\right)\sum_{i}\frac{2\tilde{g}^{i\ov{i}}}{\lambda_1^2} \left|\sum_{q}\ov{\nu_q}V_1(\ti{g}_{i\ov{q}})\right|^2 \le {} & \frac{C}{\ve \lambda_1} \sum_{i} \sum_{q} \frac{\tilde{g}^{i\ov{i}} \tilde{g}^{q\ov{q}}  |V_1 (\tilde{g}_{i\ov{q}})|^2}{\lambda_1} \\
\le {} &   \sum_{i} \sum_{q} \frac{\tilde{g}^{i\ov{i}} \tilde{g}^{q\ov{q}} |V_1 (\tilde{g}_{i\ov{q}})|^2}{\lambda_1}.
\end{split}
\end{equation}
Next, since $\sum_{\alpha>1}\mu_\alpha^2=1$,
\begin{equation} \label{tech2}
\begin{split}\left|\sum_{\alpha>1}\mu_\alpha \partial_i(\vp_{V_\alpha V_1})\right|^2&\leq\left(\sum_{\alpha>1}(\lambda_1-\lambda_\alpha)\mu_\alpha^2\right)\left(\sum_{\alpha>1}\frac{|\partial_i(\vp_{V_\alpha V_1})|^2}{\lambda_1-\lambda_\alpha}\right)\\
&=\left(\lambda_1-\sum_{\alpha>1}\lambda_\alpha\mu_\alpha^2\right)\left(\sum_{\alpha>1}\frac{|\partial_i(\vp_{V_\alpha V_1})|^2}{\lambda_1-\lambda_\alpha}\right).
\end{split}
\end{equation}
But by the definition of $W_1$, the $\lambda_{\alpha}$ and $\mu_{\alpha}$,
$$0<\ti{g}(W_1,\ov{W_1})
= 1+\frac{1}{2}(\vp_{V_1V_1}+\vp_{JV_1JV_1})
=1+ \frac{1}{2}\left(\lambda_1+\sum_{\alpha>1}\lambda_\alpha\mu_\alpha^2\right),$$
 so
$$\lambda_1-\sum_{\alpha>1}\lambda_\alpha\mu_\alpha^2\leq 2\lambda_1+2\leq (2+2\varepsilon^2)\lambda_1,$$
as long as $\lambda_1 \ge 1/\ve^2$.
Hence
\begin{equation} \label{tech25}
\begin{split}
(1-\ve^2)\left(\lambda_1-\sum_{\alpha>1}\lambda_\alpha\mu_\alpha^2\right)&\leq
2(1-\ve^2)(1+\ve^2)\lambda_1\leq 2\lambda_1,
\end{split}
\end{equation}

Finally, using $d\hat{Q}=0$ at $x_0$,
\begin{equation} \label{tech3}
\begin{split}
2\ve \sum_{i} \frac{\tilde{g}^{i\ov{i}} | \partial_i (\varphi_{V_1V_1})|^2}{\lambda_1^2}&=2\ve\sum_i \tilde{g}^{i\ov{i}}|A\vp_i+h'\de_i|\partial\varphi|^2_g|^2\\
 &\le  4\ve A^2  \sum_{i} \tilde{g}^{i\ov{i}} | \varphi_i|^2 + 4 \ve (h')^2 \sum_{i} \tilde{g}^{i\ov{i}} |\partial_i |\partial\varphi|^2_g|^2.
\end{split}
\end{equation}
Combining  (\ref{tech15}), (\ref{tech1}), (\ref{tech2}), (\ref{tech25}) and (\ref{tech3}) completes the proof of the lemma.
\end{proof}

We now complete the proof of Theorem \ref{secondord}.  Combining \eqref{LhatQ4} with Lemma  \ref{lemmauno} we have
\[
\begin{split}
0 \ge &  - 4\ve A^2  \tilde{g}^{i\ov{i}}|\varphi_i|^2  - 2 (h')^2 \sum_{i}\tilde{g}^{i\ov{i}} |\partial_i | \partial \varphi|^2_g|^2 \\ {} &
 + h' \sum_k \tilde{g}^{i\ov{i}} (| \varphi_{ik}|^2 + |\varphi_{i\ov{k}}|^2)   +h'' \tilde{g}^{i\ov{i}} |\partial_i | \partial \varphi|^2_g|^2  \\ {} & + (A -C_0) \sum_i \tilde{g}^{i\ov{i}} - An,\\
\end{split}
\]
as long as $\ve<\frac{1}{2}$ and $\lambda_1(x_0)\geq C/\ve^2$.  Recalling that $h''=2(h')^2$, and
choosing $A=C_0+2$ and $$\ve=\frac{1}{4A^2 (\sup_M | \partial \varphi |_g^2+1)},$$
we may now assume without loss of generality that $\lambda_1(x_0)\geq C/\ve^2$, and we conclude that at $x_0$ we have
$$\sum_{i}\tilde{g}^{i\ov{i}}+h' \sum_k \tilde{g}^{i\ov{i}} (| \varphi_{ik}|^2 + |\varphi_{i\ov{k}}|^2)\leq C,$$
which implies that $\lambda_1(x_0)\leq C$, completing the proof of Theorem \ref{secondord}.
\end{proof}

\section{Proof of Theorem \ref{geodesic}} \label{sectionthm1}

With the notation of the introduction,  $\Phi$ is the limit of $\Phi_\ve-|z|^2$ as $\ve\to 0$ (in $C^{1,\alpha}(X\times\Sigma)$, for any $0<\alpha<1$), where for $\ve>0$ the smooth functions $\Phi_\ve$ solve the ``$\ve$-geodesic'' equation
\begin{equation}\label{maeps}
(\pi^*\omega+\mn dz\wedge d\ov{z}+\ddbar\Phi_\ve)^{m+1}=\ve \pi^*\omega^m\wedge \mn dz\wedge d\ov{z},
\end{equation}
with the same boundary conditions as $\Phi$.
Chen \cite{Ch} (see also \cite{B2,Bo,CKNS,Gu}) has proved uniform bounds, independent of $\ve$, on $\sup_{X\times\Sigma}|\Phi_\ve|, \sup_{X\times\Sigma}|\nabla\Phi_\ve|,$ $\sup_{X\times\Sigma}\Delta\Phi_\ve$ and $\sup_{\de(X\times\Sigma)}|\nabla^2\Phi_\ve|$, where we are using the reference K\"ahler metric
 $\pi^*\omega+\mn dz\wedge d\ov{z}$ on $X\times\Sigma.$ Theorem \ref{secondord} thus applies to \eqref{maeps} and we obtain
$$\sup_{X\times\Sigma}|\nabla^2\Phi_\ve|\leq C,$$
where $C$ does not depend on $\ve$. Letting $\ve\to 0$ we conclude that $|\nabla^2\Phi|$ is in $L^\infty(X\times\Sigma)$, and so $\Phi$ is in $C^{1,1}(X\times\Sigma)$.

\end{document}